\documentclass[12pt]{amsart}
\usepackage{amscd}
\def\NZQ{\Bbb}               

\def\ZZ{{\NZQ Z}}

\def\frk{\frak}               

\def\Phi{{\frk n}}
\def\Phi{{\frk N}}
\def\opn#1#2{\def#1{\operatorname{#2}}} 
\opn\chara{char} \opn\length{\ell} \opn\pd{pd} \opn\rk{rk}
\opn\projdim{proj\,dim} \opn\injdim{inj\,dim} \opn\rank{rank}
\opn\depth{depth} \opn\sdepth{sdepth} \opn\fdepth{fdepth}
\opn\grade{grade} \opn\height{height} \opn\embdim{emb\,dim}
\opn\codim{codim}  \opn\min{min} \opn\max{max}

\opn\Tr{Tr} \opn\bigrank{big\,rank}
\opn\superheight{superheight}\opn\lcm{lcm}
\opn\trdeg{tr\,deg}
\opn\reg{reg} \opn\lreg{lreg} \opn\ini{in} \opn\lpd{lpd}
\opn\size{size}
\opn\div{div} \opn\Div{Div} \opn\cl{cl} \opn\Cl{Cl}
\opn\Spec{Spec} \opn\Supp{Supp} \opn\supp{supp} \opn\Sing{Sing}
\opn\Ass{Ass} \opn\Min{Min}
\opn\Ann{Ann} \opn\Rad{Rad} \opn\Soc{Soc}
\opn\Im{Im} \opn\Ker{Ker} \opn\Coker{Coker} \opn\Am{Am}
\opn\Hom{Hom} \opn\Tor{Tor} \opn\Ext{Ext} \opn\End{End}
\opn\Aut{Aut} \opn\id{id}  \opn\deg{deg}

\opn\nat{nat}
\opn\pff{pf}
\opn\Pf{Pf} \opn\GL{GL} \opn\SL{SL} \opn\mod{mod} \opn\ord{ord}
\opn\Gin{Gin} \opn\Hilb{Hilb}
\opn\aff{aff} \opn\con{conv} \opn\relint{relint} \opn\st{st}
\opn\lk{lk} \opn\cn{cn} \opn\core{core} \opn\vol{vol}
\opn\link{link} \opn\star{star}
\opn\gr{gr}

\def\pot#1#2{#1[\kern-0.28ex[#2]\kern-0.28ex]}

\opn\dirlim{\underrightarrow{\lim}}
\opn\inivlim{\underleftarrow{\lim}}
\let\Dirsum=\bigoplus

\let\to=\rightarrow

\def\Implies{\ifmmode\Longrightarrow \else
        \unskip${}\Longrightarrow{}$\ignorespaces\fi}
\def\implies{\ifmmode\Rightarrow \else
        \unskip${}\Rightarrow{}$\ignorespaces\fi}
\def\iff{\ifmmode\Longleftrightarrow \else
        \unskip${}\Longleftrightarrow{}$\ignorespaces\fi}

\let\:=\colon
\newtheorem{Theorem}{Theorem}[section]
\newtheorem{Lemma}[Theorem]{Lemma}
\newtheorem{Corollary}[Theorem]{Corollary}
\newtheorem{Proposition}[Theorem]{Proposition}
\newtheorem{Remark}[Theorem]{Remark}

\newtheorem{Example}[Theorem]{Example}

\newtheorem{Question}[Theorem]{Question}
\let\epsilon\varepsilon
\let\phi=\varphi
\let\kappa=\varkappa
\def\qed{\ifhmode\textqed\fi
      \ifmmode\ifinner\quad\qedsymbol\else\dispqed\fi\fi}
\def\textqed{\unskip\nobreak\penalty50
       \hskip2em\hbox{}\nobreak\hfil\qedsymbol
       \parfillskip=0pt \finalhyphendemerits=0}
\def\dispqed{\rlap{\qquad\qedsymbol}}

\opn\dis{dis}
\def\pnt{{\raise0.5mm\hbox{\large\bf.}}}

\opn\Lex{Lex}

\begin{document}

\title{\bf Computing the Stanley depth}

\author{ Dorin Popescu, Muhammad Imran Qureshi }

\thanks{The authors would like to express their gratitude to ASSMS of GC University Lahore for creating a very appropriate atmosphere for research work. This research is partially supported by HEC Pakistan. The first author was mainly supported by CNCSIS Grant  ID-PCE no. 51/2007.}

\address{Dorin Popescu, Institute of Mathematics "Simion Stoilow",
University of Bucharest, P.O.Box 1-764, Bucharest 014700, Romania}
\email{dorin.popescu@imar.ro}
\address{Muhammad Imran Qureshi, Abdus Salam School of Mathematical Sciences, GC University, Lahore,68-B New Muslim town Lahore, Pakistan}
\email{imranqureshi18@gmail.com}
\maketitle
\begin{abstract}
Let $Q$ and $Q'$ be two monomial primary ideals of a polynomial algebra $S$ over a field. We give an upper bound for the Stanley depth of $S/(Q\cap Q')$ which is reached if $Q$,$Q'$ are irreducible. Also we show that Stanley's Conjecture holds for $Q_1\cap Q_2$, $S/(Q_1\cap Q_2\cap Q_3)$, $(Q_i)_i$ being some irreducible
monomial ideals  of $S$.

  \vskip 0.4 true cm
 \noindent
\end{abstract}

\section*{Introduction}

 Let $K$ be a field and $S=K[x_1,\ldots,x_n]$ be the polynomial ring over $K$
 in $n$ variables  and $M$  a finitely generated multigraded (i.e.
$\ZZ^n$-graded) $S$-module. Given $z\in M$  a homogeneous element in
$M$ and $Z\subseteq \{x_1,\ldots,x_n\}$, let $zK[Z]\subset M$ be the
linear $K$-subspace of all elements of the form $zf$, $f\in K[Z]$.
This subspace is called Stanley space of dimension $|Z|$, if $zK[Z]$
is a free $K[Z]$-module. A Stanley decomposition of $M$ is a
presentation of the  $K$-vector space $M$ as a finite direct sum of
Stanley spaces $\mathcal{D}:\,\,M=\Dirsum_{i=1}^rz_iK[Z_i]$. Set
$\sdepth \mathcal{D}=\min\{|Z_i|:i=1,\ldots,r\}$. The number
\[
\sdepth(M):=\max\{\sdepth({\mathcal D}):\; {\mathcal D}\; \text{is a
Stanley decomposition of}\;  M \}
\]
is called the Stanley depth of $M$. This is a combinatorial invariant which has some common properties with the homological invariant depth. Stanley conjectured (see
\cite{S}) that $\sdepth\ M\geq \depth\ M$, but this conjecture is still open for a long time in spite of some  results obtained mainly for $n\leq 5$ (see \cite{A}, \cite{So}, \cite{HSY}, \cite{AP}, \cite{P}, \cite{Po}). An algorithm to compute the Stanley depth is given in \cite{HVZ} and was used here to find several examples.
Very important in our computations were the results from \cite{Bi}, \cite{C2} and \cite{Sh}.

Let  $Q,Q'$ be two monomial primary ideals such that $\dim S/(Q+Q')=0$. Then $\sdepth\ S/(Q\cap Q')\leq $
$$\max\{\min\{\dim(S/Q'),\lceil\frac{\dim(S/Q)}{2}\rceil\},\\  \min\{\dim(S/Q),\lceil\frac{\dim(S/Q')}{2}\rceil\}\},$$
and the bound is reached when $Q,Q'$ are non-zero irreducible monomial ideals (see Proposition \ref{up}, or more general in Corollary \ref{eg} ), $\lceil\frac{a}{2}\rceil$
being the smallest integer $\geq a/2$, $a\in {\bf Q}$.

Let $Q_1,Q_2,Q_3$ be three non-zero irreducible monomial ideals of $S$. If $\dim S/(Q_1+Q_2)=0$ then $$\sdepth (Q_1\cap Q_2)\geq
\lceil\frac{\dim(S/Q_1)}{2}\rceil+\lceil\frac{\dim(S/Q_2)}{2}\rceil$$ (see Lemma \ref{lb} , or more general in Theorem \ref{Lob}).
In this case, our bound is better than the bound given by \cite{KY} and \cite{O} (see Remark \ref{ky}).
 Using these results we show that
$\sdepth\ (Q_1\cap Q_2)\geq \depth\ (Q_1\cap Q_2)$, and
$$\sdepth\ S/(Q_1\cap Q_2\cap Q_3)\geq \depth\ S/(Q_1\cap Q_2\cap Q_3),$$ that is
Stanley's Conjecture holds for $Q_1\cap Q_2$ and $S/(Q_1\cap Q_2\cap Q_3)$ (see Theorem \ref{sci} , Theorem \ref{s3}).
\vskip 1 cm

\section{A lower bound for Stanley's depth of some cycle modules}
We start with few simple lemmas which we include for the completeness of our paper.
\begin{Lemma}\label{l1}
Let $Q$ be a monomial primary ideal in $S=K[x_1,\ldots ,x_n]$. Suppose that $\sqrt{Q}=(x_1,\ldots ,x_r)$ where $1\leq r\leq n$, Then
there exists a Stanley decomposition
$$S/Q=\oplus uK[x_{r+1},\ldots ,x_n],$$
where the sum runs on monomials $u\in K[x_1,\ldots ,x_r]\setminus (Q\cap K[x_1,\ldots ,x_r])$.
\end{Lemma}
\begin{proof}
Given $u,v \in K[x_1,\ldots,x_r]\setminus (Q\cap K[x_1,\ldots,x_r])$ and $h,g \in K[x_{r+1},\ldots,x_n]$  with $uh=vg$  then we get $u=v$, $g=h$. Thus the given sum is direct.
 Note that there exist just a finite number of monomials in $K[x_1,\ldots ,x_r]\setminus (Q\cap K[x_1,\ldots ,x_r])$. Let $0\neq \alpha \in (S\setminus Q)$ be a monomial.
 Then $\alpha=uf$, where $f\in K[x_{r+1},\ldots,x_n]$ and $u\in K[x_1,\ldots,x_r]$.
Since $\alpha\not \in Q$ we have $u\not\in Q$. Thus $S/Q\subset \oplus uK[x_{r+1},\ldots,x_n]$, the other inclusion being trivial.
\end{proof}
\begin{Lemma}\label{l2}
Let  $Q$ be a monomial primary ideal in $S=K[x_1,\ldots,x_n]$.  Then
$\sdepth S/Q= \dim S/Q=\depth S/Q$.
\end{Lemma}
\begin{proof}
Let $\dim S/Q=n-r$ for some $0\leq r\leq n$.
We have $\dim S/Q\geq \sdepth S/Q$ by \cite[Theorem 2.4]{A}. Renumbering variables we may suppose that $\sqrt{Q}=(x_1,\ldots ,x_r)$.
Using the above lemma we get the converse inequality.
As $S/Q$ is  Cohen Macaulay it follows
$\dim S/Q=\depth S/Q$, which is enough.
\end{proof}
\begin{Lemma}\label{l3}
Let $I$,$J$ be two monomial ideals of $S=K[x_1,\ldots,x_n]$. Then
$$\sdepth(S/(I\cap J))\geq
\max\{\min\{\sdepth(S/I),\sdepth(I/(I\cap J))\},$$ $$\min\{\sdepth(S/J),\sdepth(J/(I\cap J))\}\}.$$
\end{Lemma}
\begin{proof}
Consider the following exact sequence of $S$-modules.
$$0\rightarrow I/(I\cap J) \to S/(I\cap J) \to S/I \to 0.$$
By  \cite[Lemma 2.2]{R}, we have
\begin{equation}
\label{1}
\sdepth(S/(I\cap J))\geq \min\{\sdepth(S/I),\sdepth(I/(I\cap J))\}.
\end{equation}
Similarly, we get
\begin{equation}\label{2}
\sdepth(S/(I\cap J))\geq \min\{\sdepth(S/J),\sdepth(J/(I\cap J))\}.
\end{equation}
The proof ends using (\ref{1}) and (\ref{2}).
\end{proof}
\begin{Proposition}\label{low}
Let  $Q$, $Q'$ be two monomial primary ideals in $S=K[x_1,\ldots,x_n]$ with different associated prime ideals. Suppose that
  $\sqrt{Q}=(x_1,\ldots ,x_t)$,  $\sqrt{Q'}=(x_{r+1},\ldots ,x_n)$ for some integers $t,r$ with $0\leq r\leq t\leq n$.
 Then $\sdepth(S/(Q\cap Q'))\geq$
\begin{multline*}
\max\{\min_v \{r, \sdepth(Q'\cap K[x_{t+1},\ldots, x_n]),\sdepth((Q':v)\cap K[x_{t+1},\ldots, x_n])\}, \\ \;\;\;\;\;\;\;
\min_w \{n-t, \sdepth(Q\cap K[x_{1},\ldots, x_r]),\sdepth((Q:w)\cap K[x_{1},\ldots, x_r])\}\},
\end{multline*}
\\where $v$,$w$ run in the set of monomials  containing only variables from $\{x_{r+1},\ldots,x_t\}$,  $w\not\in Q$, $v\not\in Q'$.
\end{Proposition}
\begin{proof}
If $Q$, or $Q'$ is zero then  the inequality holds trivially. If $r=0$  then \\
$Q\cap  K[x_{1},\ldots, x_r]=Q\cap K=0$,  and  the  inequality is clear. A similar case is $t=n$. Thus we may suppose $1\leq r\leq t<n$.
Applying Lemma \ref{l3} it is enough to show that $\sdepth(Q'/(Q\cap Q'))\geq$
$$\min \{\sdepth(Q'\cap K[x_{t+1},\ldots x_n]),\sdepth((Q':v)\cap K[x_{t+1},\ldots x_n])\},$$
where $v$ is a monomial of $K[x_{r+1},\ldots, x_n]\setminus (Q\cap Q')$.  We have a canonical injective map $Q'/(Q\cap Q')\to S/Q$.
By Lemma \ref{l1} we get
$$Q'/(Q\cap Q')=Q'\cap (\oplus uK[x_{t+1},\ldots,x_n])=\oplus (Q'\cap uK[x_{t+1},\ldots,x_n]),$$
where $u$ runs in the monomials of $K[x_1,\ldots,x_{t}]\setminus Q$.
Here
$Q'\cap uK[x_{t+1},\ldots,x_n]=u(Q'\cap K[x_{t+1},\ldots,x_n])$ if $u\in K[x_1,\ldots,x_r]$
and $Q'\cap uK[x_{t+1},\ldots,x_n]=u((Q':u)\cap K[x_{t+1},\ldots,x_n])$ if $u\not\in K[x_1,\ldots,x_r]$.
If $u\in Q'$ then $Q':u=S$. We have $$Q'/(Q\cap Q')=
(\oplus u(Q'\cap K[x_{t+1},\ldots,x_n]))\oplus$$
$$ (\oplus z K[x_{t+1},\ldots,x_n])\oplus (\oplus uv((Q':v)\cap K[x_{t+1},\ldots,x_n])),$$
where the sum runs for all monomials  $u\in (K[x_1,\ldots,x_r]\setminus Q)$ ,  $z\in Q'\setminus Q$ and $v\in K[x_{r+1},\ldots,x_t]$, $v\not\in Q'\cup Q$.
Now it is enough to apply \cite[Lemma 2.2]{R} to get the above inequality.
\end{proof}

\begin{Theorem}\label{Low}
Let $Q$ and $Q'$ be two irreducible monomial ideals of $S$. Then $\sdepth_S S/(Q\cap Q')\geq$
\begin{multline*}
 \max\{\min\{\dim(S/Q'),\lceil\frac{\dim(S/Q)+\dim(S/(Q+Q'))}{2}\rceil\},\\
 \min\{\dim(S/Q),\lceil\frac{\dim(S/Q')+\dim(S/(Q+Q'))}{2}\rceil\}\}.
\end{multline*}
\end{Theorem}
\begin{proof}If the associated prime ideals of $Q,Q'$ are the same then the above inequality says that $\sdepth_S S/(Q\cap Q')\geq \dim S/Q$, which follows from Lemma \ref{l2}. Thus we may suppose that the associated prime ideals of $Q,Q'$ are different.
We may suppose that $Q$ is generated in variables $\{x_1,\ldots,x_t\}$ and $Q'$ is generated in variables $\{x_{r+1},\ldots,x_p\}$ for some integers $0\leq r\leq t\leq p\leq n$. Since  $\dim(S/Q)=n-t$, $\dim(S/Q')=n-p+r$ and $\dim(S/(Q+ Q'))=n-p$ we get
$$n-t-\lfloor\frac{p-t}{2}\rfloor=\lceil\frac{(n-t)+(n-p)}{2}\rceil=\lceil\frac{\dim(S/Q)+\dim(S/(Q+Q'))}{2}\rceil,$$
$\lfloor\frac{a}{2}\rfloor$
being the biggest integer $\leq a/2$, $a\in {\bf Q}.$
Similarly, we have
$$n-p+r-\lfloor\frac{r}{2}\rfloor=\lceil\frac{\dim(S/Q')+\dim(S/(Q+Q'))}{2}\rceil.$$
On the other hand by \cite{C2}, and \cite[Theorem 2.4]{Sh} $\sdepth(Q'\cap K[x_{t+1},\ldots,x_n])=n-t-\lfloor\frac{p-t}{2}\rfloor$
and $\sdepth(Q\cap K[x_1,\ldots,x_r,x_{p+1},\ldots,x_n])=n-p+r-\lfloor\frac{r}{2}\rfloor$.
In fact, the quoted result  says in particular that sdepth of each irreducible ideal $L$ depends only on the number of variables of the ring and the number of variables generating $L$ (a description of irreducible monomial ideals is given in \cite{Vi}).
Since $(Q':v)\cap K[x_{t+1},\ldots,x_n]$ is still an irreducible ideal generated by the same variables as $Q'$ we conclude that $$\sdepth((Q':v)\cap K[x_{t+1},\ldots,x_n])=\sdepth(Q'\cap K[x_{t+1},\ldots,x_n]),$$ $v\not \in Q'$ being any monomial.
Similarly, $$\sdepth((Q:w)\cap K[x_1,\ldots,x_r,x_{p+1},\ldots,x_n])=\sdepth(Q\cap K[x_1,\ldots,x_r,x_{p+1},\ldots,x_n]).$$
It follows that our inequality holds if $p=n$ by Proposition \ref{low}.

 Set $S'=K[x_1,\ldots,x_p]$, $q=Q\cap S'$, $q'=Q'\cap S'$. As above (case $p=n$) we get
$\sdepth_{S'} S'/(q\cap q')\geq$
$$\max\{\min\{\dim(S'/q'),\lceil\frac{\dim(S'/q)}{2}\rceil\},\min\{\dim(S'/q),\lceil\frac{\dim(S'/q')}{2}\rceil\}\}=$$
$$\max\{\min\{r,\lceil\frac{p-t}{2}\rceil\},\min\{p-t,\lceil\frac{r}{2}\rceil\}\}.$$
Using   \cite[Lemma 3.6]{HVZ}, we have
$$\sdepth_S (S/(Q\cap Q'))=\sdepth_S (S/(q\cap q')S)=n-p+\sdepth_{S'} (S'/(q\cap q')).$$
It follows that
$$\sdepth_S (S/(Q\cap Q'))\geq n-p+\max\{\min\{r,\lceil\frac{p-t}{2}\rceil\},\min\{p-t,\lceil\frac{r}{2}\rceil\}\}=$$
$$\max\{\min\{n-p+r,n-p+\lceil\frac{p-t}{2}\rceil\},\min\{n-t,n-p+\lceil\frac{r}{2}\rceil\}\}=$$
$$\max\{\min\{n-p+r,n-t-\lfloor\frac{p-t}{2}\rfloor\},\min\{n-t,n-p+r-\lfloor\frac{r}{2}\rfloor\}\},$$
which is enough.
\end{proof}

\section{An upper bound for Stanley's depth of some cycle modules}

Let $Q$,$Q'$ be two monomial primary ideals of $S$. Suppose that $Q$ is generated in variables $\{x_1,\ldots,x_t\}$ and $Q'$ is generated in variables $\{x_{r+1},\ldots,x_n\}$ for some integers $1\leq r\leq t<n$.
Thus the prime ideals associated to $Q\cap Q'$ have dimension $\geq 1$ and it follows $\depth(S/(Q\cap Q'))\geq 1$. Then $\sdepth(S/(Q\cap Q'))\geq 1$ by \cite[Corollary 1.6]{C1}, or \cite[Theorem 1.4]{C3}.
Let $\mathcal{D}: \;\;\;\ S/(Q\cap Q')= \oplus _{i=1} ^s u_i K[Z_i]$
be a Stanley decomposition of $S/(Q\cap Q')$ with $\sdepth \mathcal{D}=\sdepth(S/(Q\cap Q'))$. Thus $|Z_i|\geq 1$ for all $i$.
Renumbering $(u_i,Z_i)$ we may suppose that $1\in u_1 K[Z_1]$, so $u_1 =1$.
Note that $Z_i$  cannot have mixed variables from $\{x_1,\ldots,x_r\}$ and $\{x_{t+1},\ldots,x_n\}$ because otherwise $u_i K[Z_i]$ will be not a free $K[Z_i]$-module. As $\mid Z_1\mid \geq 1$ we may have either $Z_1\subset \{x_1,\ldots,x_r\}$ or $Z_1\subset \{x_{t+1},\ldots,x_n\}.$
\begin{Lemma}\label{l4}
Suppose $Z_1\subset \{x_1,\ldots,x_r\}$. Then $\sdepth(\mathcal{D})\leq \min\{r,\lceil\frac{n-t}{2}\rceil\}.$
\end{Lemma}
\begin{proof}
Clearly $\sdepth(\mathcal{D}) \leq \mid Z_1 \mid \leq r$. Let $ a \in\mathbb{N}$ be such that $x_i ^a \in Q'$
for all $t<i\leq n $. Let $T=K[y_{t+1},\ldots,y_n]$ and $\varphi : T\longrightarrow S$ be the $K$-morphism given by
$y_i\longrightarrow x_i ^a$.
 The composition map $\psi$:
$T\longrightarrow S\longrightarrow S/(Q\cap Q')$
is injective. Note also that we may consider $Q'\cap K[x_{t+1},\ldots,x_n]\subset S/(Q\cap Q')$ since $Q\cap K[x_{t+1},\ldots,x_n]=0$. We have
\begin{multline*}
(y_{t+1},\ldots,y_n)=\psi^{-1}(Q'\cap K[x_{t+1},\ldots,x_n])=\oplus \psi^{-1}(u_j K[Z_j]\cap Q'\cap K[x_{t+1},\ldots,x_n]).
\end{multline*}
If $u_j K[Z_j]\cap Q'\cap K[x_{t+1},\ldots,x_n]\neq 0$ then $u_j\in K[x_{t+1},\ldots,x_n]$.
Also we have $Z_j\subset \{x_{t+1},\ldots,x_n\}$, otherwise $u_j K[Z_j]$ is not free over $K[Z_j]$. Moreover,
\\if $\psi^{-1}(u_j K[Z_j]\cap Q'\cap K[x_{t+1},\ldots,x_n])\neq 0$ then $u_j = x_{t+1} ^{b_{t+1}} \ldots x_n ^{b_n}$, $b_i\in \mathbb{N}$ is such that if $x_i\not\in Z_j$, $t<i\leq n$, then $a\mid b_i$, let us say $b_i=a c_i$ for some $c_i\in \mathbb{N}$ . Denote $c_i= \lceil \frac{b_i}{a}\rceil$ when $x_i \in Z_j$. We get
$$\psi^{-1}(u_j K[Z_j]\cap Q'\cap K[x_{t+1},\ldots,x_n])=y_{t+1} ^{c_{t+1}} \ldots y_n ^{c_n} K[V_j],$$
where $V_j=\{y_i:t<i\leq n,x_i\in Z_j\}$. Thus $\psi^{-1}(u_j K[Z_j]\cap Q'\cap K[x_{t+1},\ldots,x_n])$ is a Stanley space of $T$ and so $\mathcal{D}$ induces a Stanley decomposition $\mathcal{D}'$ of $(y_{t+1},\ldots,y_n)$ such that $\sdepth(\mathcal{D})\leq \sdepth(\mathcal{D}')\leq\sdepth(y_{t+1},\ldots,y_n)$ because $\mid Z_j\mid=\mid V_j\mid$.
Consequently $\sdepth(\mathcal{D})\leq \lceil\frac{n-t}{2}\rceil$ by  \cite{Bi} and so
 $\sdepth(\mathcal{D})\leq \min\{r,\lceil\frac{n-t}{2}\rceil\}$.
 
 Note also that if $t=n$, or $r=0$ then the same proof works; so $\sdepth S/(Q\cap Q')=0$, which is clear because $\depth S/(Q\cap Q')=0$ (see 
 \cite[Corollary 1.6]{C1}).
\end{proof}
\begin{Proposition}\label{up}
Let $Q$,$Q'$ be two non-zero monomial primary ideals of $S$ with different associated prime ideals. Suppose that $\dim(S/(Q+ Q'))=0$. Then \\
$\sdepth _S (S/(Q\cap Q'))\leq$
 $$\max\{\min\{\dim(S/Q'),\lceil\frac{\dim(S/Q)}{2}\rceil\},\min\{\dim(S/Q),\lceil\frac{\dim(S/Q')}{2}\rceil\}\}.$$
\end{Proposition}
\begin{proof}
If one of $Q$,$Q'$ is of dimension zero then $\depth(S/(Q\cap Q'))=0$ and so by  \cite[Corollary 1.6]{C1} (or \cite[Theorem 1.4]{C3}) $\sdepth(S/(Q\cap Q'))=0$, that is the inequality holds trivially. Thus we may suppose  after renumbering of variables that $Q$ is generated in variables $\{x_1,\ldots,x_t\}$ and $Q'$ is generated in variables $\{x_{r+1},\ldots,x_p\}$ for some integers $t,r,p$ with $1\leq r\leq t< p\leq n$, or $0\leq r<t\leq n$. By hypothesis we have $p=n$.
Let $\mathcal{D}$ be the Stanley decomposition of $S/(Q\cap Q')$ such that $\sdepth(\mathcal{D})=\sdepth(S/(Q\cap Q'))$. Let $Z_1$ be defined as in Lemma \ref{l4}, that is $K[Z_1]$ is the Stanley space corresponding to $1$. If $Z_1\subset \{x_1,\ldots,x_r\}$ then by Lemma \ref{l4} $$\sdepth(\mathcal{D})\leq \min\{r,\lceil\frac{n-t}{2}\rceil\}=\min\{\dim(S/Q'),\lceil\frac{\dim(S/Q)}{2})\rceil\}.$$
If $Z_1\subset \{x_{t+1},\ldots,x_n\}$ we get analogously $$\sdepth(\mathcal{D})\leq \min\{n-t,\lceil\frac{r}{2}\rceil\}=\min\{\dim(S/Q),\lceil\frac{\dim(S/Q')}{2})\rceil\},$$
which shows our inequality.
\end{proof}
\begin{Theorem}\label{Up}
Let $Q$ and $Q'$ be two non-zero monomial primary ideals of $S$ with different associated  prime ideals. Then  $\sdepth_S S/(Q\cap Q')\leq$
\begin{multline*}
 \max\{\min\{\dim(S/Q'),\lceil\frac{\dim(S/Q)+\dim(S/(Q+Q'))}{2}\rceil\},\\
 \min\{\dim(S/Q),\lceil\frac{\dim(S/Q'+\dim(S/(Q+Q')))}{2}\rceil\}\}.
\end{multline*}
\end{Theorem}
\begin{proof}
As in the proof of Proposition \ref{up} we may suppose that $Q$ is generated in variables $\{x_1,\ldots,x_t\}$ and $Q'$ is generated in variables $\{x_{r+1},\ldots,x_p\}$ for some integers $1\leq r\leq t< p\leq n$, or $0\leq r<t\leq n$ but now we have not in general $p=n$. Set $S'=K[x_1,\ldots,x_p]$, $q=Q\cap S'$, $q'=Q'\cap S'$. Using Proposition \ref{up} we get
$$\sdepth _S (S/(q\cap q'))\leq \max\{\min\{\dim(S/q'),\lceil\frac{\dim(S/q)}{2}\rceil\},$$
$$\min\{\dim(S/q),\lceil\frac{\dim(S/q')}{2}\rceil\}\}.$$
By \cite[Lemma 3.6]{HVZ} we have $$\sdepth_S (S/(Q\cap Q'))=\sdepth_S (S/(q\cap q')S)=n-p+\sdepth_{S'} (S'/(q\cap q')).$$
As in the proof of Theorem \ref{Low}, it  follows that
$$\sdepth_S (S/(Q\cap Q'))\leq n-p+\max\{\min\{r,\lceil\frac{p-t}{2}\rceil\},\min\{p-t,\lceil\frac{r}{2}\rceil\}\}=$$
$$\max\{\min\{n-p+r,n-t-\lfloor\frac{p-t}{2}\rfloor\},\min\{n-t,n-p+r-\lfloor\frac{r}{2}\rfloor\}\},$$
which is enough.
\end{proof}
\begin{Corollary}\label{eg}
Let $Q$ and $Q'$ be two non-zero monomial irreducible ideals of $S$ with different associated  prime ideals. Then  $\sdepth_S S/(Q\cap Q')=$
\begin{multline*}
 \max\{\min\{\dim(S/Q'),\lceil\frac{\dim(S/Q)+\dim(S/(Q+Q'))}{2}\rceil\},\\  \min\{\dim(S/Q),\lceil\frac{\dim(S/Q')+\dim(S/(Q+Q'))}{2}\rceil\}\}.
\end{multline*}
\end{Corollary}
For the proof apply Theorem \ref{Low} and Theorem \ref{Up}.
\begin{Corollary}\label{pr}
Let $P$ and $P'$ be two different non-zero monomial prime ideals of $S$, which are not included one in the other. Then  $\sdepth_S S/(P\cap P')=$
\begin{multline*}
 \max\{\min\{\dim(S/P'),\lceil\frac{\dim(S/P)+\dim(S/(P+P'))}{2}\rceil\},\\  \min\{\dim(S/P),\lceil\frac{\dim(S/P')+\dim(S/(P+P'))}{2}\rceil\}\}.
\end{multline*}
\end{Corollary}
\begin{proof}
For the proof apply  Corollary \ref{eg}.
\end{proof}
\begin{Corollary}
Let $\triangle$ be a simplicial complex  in $n$ vertices with only two different facets $F$, $F'$. Then
$\sdepth K[\triangle]=$
\begin{multline*}
 \max\{\min\{|F'|,\lceil\frac{|F|+|F\cap F'|}{2}\rceil\},  \min\{|F|,\lceil\frac{|F'|+|F\cap F'|}{2}\rceil\}\}.
\end{multline*}
\end{Corollary}

\vskip 1 cm
\section{An Illustration}

Let $S=K[x_1,\ldots,x_6]$, $Q=(x_1^2,x_2^2,x_3^2,x_4^2,x_1x_2x_4,x_1x_3x_4)$, $Q'=(x_4^2,x_5,x_6)$.  By our Theorem \ref{Up} we get $$\sdepth
S/(Q\cap Q')\leq \max\{\min\{3,\lceil\frac {2}{2}\rceil\}, \min\{2,\lceil\frac {3}{2}\rceil\}=\max\{1,2\}=2.$$
On the other hand,  we claim that $I=((Q:w)\cap K[x_1,x_2,x_3])=(x_1^2,x_2^2,x_3^2,x_1x_2,x_1x_3)$  for $w=x_4$ and
 $\sdepth I=1<2=\sdepth(Q\cap K[x_1,x_2,x_3])$. Thus our Proposition \ref{low} gives $$\sdepth
S/(Q\cap Q')\geq \max\{\min\{3,\lceil\frac {2}{2}\rceil\}, \min\{2,\lceil\frac {3}{2}\rceil,1\}=1.$$
In this section, we will show that  $\sdepth(S/(Q\cap Q'))=1$.

First we prove our claim. Suppose that there exists a Stanley decomposition $\mathcal{D}$ of $I$ with  $\sdepth \mathcal{D}\geq 2$. Among the Stanley  spaces of
 $\mathcal{D}$ we have five important $x_1^2K[Z_1]$, $x_2^2K[Z_2]$, $x_3^2K[Z_3]$, $x_1x_2K[Z_4]$, $x_1x_3K[Z_5]$ for some subsets $Z_i\subset \{x_1,x_2,x_3\}$ with
 $|Z_i|\geq 2$. If  $Z_4=\{x_1,x_2,x_3\}$  and $Z_5$ contains $x_2$ then the last two Stanley spaces will have a non-zero intersection and if $Z_1$ contains $x_2$ then the first and the fourth Stanley space will have non-zero intersection. Now if $x_2\not\in Z_5$ and $x_2\not\in Z_1$ then the first and the last space  will intersect. Suppose that $Z_4=\{x_1,x_2\}$. Then $x_2\not
 \in Z_1$ (resp. $x_1\not \in Z_2$) because otherwise the intersection of  $x_1x_2K[Z_4]$ with the first Stanley space (resp. the second one) will be again non-zero. As $|Z_1|,|Z_2|\geq 2$ we
 get $Z_1=\{x_1,x_3\}$, $Z_2=\{x_2,x_3\}$. But $x_1\not \in Z_3$ because otherwise the first and the third Stanley space will contain $x_1^2x_3^2$, which is impossible. Similarly, $x_2\not \in Z_3$, which contradicts  $|Z_3|\geq 2$. The case  $Z_5=\{x_1,x_3\}$ gives a similar contradiction.

 Now suppose that $Z_4=\{x_1,x_3\}$. If $Z_5\supset \{x_1,x_2\}$ we see that the intersection of the last two Stanley spaces  from the above five,  contains $x_1^2x_2x_3$
 and if $Z_5=\{x_2,x_3\}$ we see that the intersection of the same Stanley spaces  contains $x_1x_2x_3$. Contradiction (we saw that $Z_5\not =\{x_1,x_3\}$)! Hence $\sdepth \mathcal{D}\leq 1$ and so
 $\sdepth I=1$ using \cite{C1}.

Next we show that $\sdepth S/(Q\cap Q')=1$. Suppose that $\mathcal{D}'$ is  a Stanley decomposition of $S/(Q\cap Q')$ such that $\sdepth
 S/(Q\cap Q')=2$. We claim that $\mathcal{D}'$ has the form
$$ S/(Q\cap Q')=(\oplus vK[x_5,x_6])  \oplus    ( \oplus _{i=1} ^s u_i K[Z_i] )   $$
for some monomials $v\in (K[x_1,\ldots,x_4]\setminus Q)$, $u_i\in (Q\cap K[x_1,\ldots,x_4])$ and $Z_i\subset \{x_1,x_2,x_3\}$.
Indeed, let $v\in (K[x_1,\ldots,x_4]\setminus Q)$. Then $vx_5$, $vx_6$ belong to some Stanley spaces of $\mathcal{D}'$, let us say $uK[Z]$, $u'K[Z']$. The presence of $x_5$ in $u$ or $Z$ implies that $Z$ does not contain any $x_i$, $1\leq i\leq 3$, otherwise $uK[Z]$ will be not free over $K[Z]$. Thus
$Z\subset \{x_5,x_6\}$. As $|Z|\geq 2$ we get $Z= \{x_5,x_6\}$ and similarly  $Z'= \{x_5,x_6\}$. Thus $vx_5x_6\in  (uK[Z]\cap  u'K[Z'])$ and it follows that
$u=u'$, $Z=Z'$ because the  sum in $\mathcal{D}'$ is direct. It follows that $u|vx_5$, $u|vx_6$ and so $u|v$, that is $v=uf$, $f$ being a monomial in $x_5,x_6$.
As  $v\in K[x_1,\ldots,x_4]$ we get $f=1$ and so $u=v.$

A monomial $w\in (Q\setminus Q')$ is not a multiple of $x_5,x_6$, because otherwise $w\in Q'$. Suppose $w$ belongs to a Stanley space $uK[Z]$ of
 $\mathcal{D}'$. If $u\in (K[x_1,\ldots,x_4]\setminus Q)$ then  as above $\mathcal{D}'$  has also a Stanley space $uK[x_5,x_6]$ and both spaces contains $u$. This is false since the sum is direct. Thus $u\in (Q\cap K[x_1,\ldots,x_4])$, which shows our claim.

Hence $\mathcal{D}'$ induces two Stanley decompositions
$S/Q=\oplus_{v\in (K[x_1,\ldots,x_4]\setminus Q)} vK[x_5,x_6]$, $Q/(Q\cap Q')= \oplus _{i=1} ^s u_i K[Z_i] $, where $u_i\in (Q\cap K[x_1,\ldots,x_4])$
and $Z_i\subset \{x_1,x_2,x_3\}$. Then we get the following Stanley decompositions $$Q\cap K[x_1,\ldots,x_3]= \oplus _{i=1, u_i\not \in (x_4)} ^s u_i K[Z_i],\ \
I=\oplus _{i=1, x_4 | u_i} ^s (u_i/x_4) K[Z_i].$$ As $2\leq \min_i|Z_i|$ we get $ \sdepth I\geq 2$. Contradiction!

\vskip 1 cm
\section{A lower bound for Stanley's depth of some ideals}

Let $Q$, $Q'$ be two non-zero irreducible monomial ideals of $S$ such that $\sqrt{Q}=(x_1,\ldots,x_t)$, $\sqrt{Q'}=(x_{r+1},\ldots,x_p)$
for some integers $r,t,p$ with $1\leq r\leq t<p\leq n$, or  $0= r< t< p\leq n$, or $1\leq r\leq t=p\leq n$.
\begin{Lemma}\label{ea}
Suppose that  $p=n$, $t=r$. Then $$\sdepth (Q\cap Q')\geq \lceil\frac{r}{2}\rceil+\lceil\frac{n-r}{2}\rceil\geq n/2.$$
\end{Lemma}
\begin{proof} It follows $1\leq r<p$. Let $f\in Q\cap K[x_1,\ldots,x_r]$, $g\in  Q'\cap K[x_{r+1},\ldots,x_n]$ and ${\mathcal M}(T)$ be the monomials from an ideal $T$.
The correspondence $(f,g)\to fg$ defines a map $\varphi: {\mathcal M}(Q\cap K[x_1,\ldots,x_r])\times {\mathcal M}(Q'\cap K[x_{r+1},\ldots,x_n])\to {\mathcal M}(Q\cap Q')$, which is
injective. If $w$ is a monomial of $Q\cap Q'$, let us say $w=fg$ for some monomials $f\in K[x_1,\ldots,x_r]$, $g\in K[x_{r+1},\ldots,x_n]$
then $fg\in Q$ and so $f\in Q$ because the variables $x_i$, $i>r$ are regular on $S/Q$. Similarly, $g\in Q'$ and so $w=\varphi((f,g))$, that is $\varphi$
is surjective. Let $\mathcal{D}$ be a Stanley decomposition of $Q\cap K[x_1,\ldots,x_r]$,
$$\mathcal{D}:\ \ \  Q\cap K[x_1,\ldots,x_r]=\oplus_{i=1}^s u_iK[Z_i]$$ with $\sdepth \mathcal{D}=\sdepth  (Q\cap K[x_1,\ldots,x_r])$ and
$\mathcal{D'}$  a Stanley decomposition of $Q'\cap K[x_{r+1},\ldots,x_n]$,
$$\mathcal{D'}:\ \ \  Q'\cap K[x_{r+1},\ldots,x_n]=\oplus_{j=1}^e v_jK[T_j]$$ with $\sdepth \mathcal{D'}=\sdepth  (Q'\cap K[x_{r+1},\ldots,x_n]).$
They induce a Stanley decomposition
$$\mathcal{D''}:  \ \ \ \ Q\cap Q'=\oplus_{j=1}^e \oplus_{i=1}^s u_iv_jK[Z_i\cup T_j]$$
because of the bijection $\varphi$. Thus
$$\sdepth (Q\cap Q')\geq \sdepth \mathcal{D''}=\min_{i,j} (|Z_i|+|T_j|)\geq \min_{i}|Z_i| +\min_{j}|T_j|=$$
$$\sdepth \mathcal{D}+\sdepth \mathcal{D'}=\sdepth (Q\cap K[x_1,\ldots,x_r])+\sdepth  (Q'\cap K[x_{r+1},\ldots,x_n])=$$
$$(r-\lfloor\frac{r}{2}\rfloor)+(n-r-\lfloor\frac{n-r}{2}\rfloor)=
\lceil\frac{r}{2}\rceil+\lceil\frac{n-r}{2}\rceil\geq n/2.$$
\end{proof}

\begin{Remark}\label{ky} {\em Suppose that $n=8$, $r=1$. Then by the above lemma we get $\sdepth (Q\cap Q')\geq \lceil\frac{1}{2}\rceil+\lceil\frac{7}{2}\rceil=5$. Since
$|G(Q\cap Q')|=7$ we get by \cite{KY}, \cite{O} the same lower bound $\sdepth (Q\cap Q')\geq 8-\lfloor\frac{7}{2}\rfloor =5$.
If  $n=8$, $r=2$ then by \cite{KY}, \cite{O} we have
$\sdepth (Q\cap Q')\geq 8-\lfloor\frac{12}{2}\rfloor =2$ but our previous lemma gives $\sdepth (Q\cap Q')\geq \lceil\frac{2}{2}\rceil+\lceil\frac{6}{2}\rceil=4.$}
\end{Remark}

\begin{Lemma}\label{lb}
Suppose that  $p=n$. Then $$\sdepth (Q\cap Q')\geq \lceil\frac{r}{2}\rceil+\lceil\frac{n-t}{2}\rceil.$$
\end{Lemma}
\begin{proof} We show that
$$Q\cap Q'=(Q\cap Q'\cap K[x_{r+1},\ldots,x_t])S\oplus $$
$$(\oplus_w w(((Q\cap Q'):w)\cap K[x_1,\ldots,x_r,x_{t+1},\ldots,x_n])),$$
where $w$ runs in the monomials of $K[x_{r+1},\ldots,x_t]\setminus (Q\cap Q')$. Indeed, a monomial $h$ of $S$ has the form $h=fg$ for some monomials
$f\in K[x_{r+1},\ldots,x_t]$, $g\in  K[x_1,\ldots,x_r,x_{t+1},\ldots,x_n]$. Since $Q, Q'$ are irreducible we see that $h\in Q\cap Q'$
either when $f$ is a multiple of a minimal generator  of $Q\cap Q'\cap K[x_{r+1},\ldots,x_t]$, or $f\not \in (Q\cap Q'\cap K[x_{r+1},\ldots,x_t])$  and then $h\in f(((Q\cap Q'):f)\cap K[x_1,\ldots,x_r,x_{t+1},\ldots,x_n])$.

Let $\mathcal{D}$ be a Stanley decomposition of $(Q\cap Q'\cap K[x_{r+1},\ldots,x_t])S$,
$$\mathcal{D}:\ \ \  (Q\cap Q'\cap K[x_{r+1},\ldots,x_t])S=\oplus_{i=1}^s u_iK[Z_i]$$ with $\sdepth \mathcal{D}=\sdepth  (Q\cap Q'\cap K[x_{r+1},\ldots,x_t])S$ and for all \\
$w\in (K[x_{r+1},\ldots,x_t]\setminus (Q\cap Q')),$
let $\mathcal{D}_w$ be a Stanley decomposition of \\
$((Q\cap Q'):w)\cap K[x_1,\ldots,x_r,x_{t+1},\ldots,x_n]$,
$$\mathcal{D}_w:\ \ \  ((Q\cap Q'):w)\cap K[x_1,\ldots,x_r,x_{t+1},\ldots,x_n]=\oplus_w\oplus_{j} v_{wj}K[T_{wj}]$$ with $\sdepth \mathcal{D}_w=\sdepth  (((Q\cap Q'):w)\cap K[x_1,\ldots,x_r,x_{t+1},\ldots,x_n]).$
Since
\\ $K[x_{r+1},\ldots,x_t]\setminus  (Q\cap Q')$ contains just a finite set of monomials we get  a Stanley decomposition of $Q\cap Q'$,
$$\mathcal{D'}:\ \ \ Q\cap Q'=(\oplus_{i=1}^s u_iK[Z_i])\oplus (\oplus_w\oplus_{j} wv_{wj}K[T_{wj}]),$$
where $w$ runs in the monomials of  $K[x_{r+1},\ldots,x_t]\setminus (Q\cap Q')$. Then
$$\sdepth \mathcal{D'}=\min_w\{\sdepth \mathcal{D},\sdepth \mathcal{D}_w\}=
\min_w\{\sdepth  (Q\cap Q'\cap K[x_{r+1},\ldots,x_t])S,$$ $$\sdepth  (((Q\cap Q'):w)\cap K[x_1,\ldots,x_r,x_{t+1},\ldots,x_n])\}.$$
But $((Q\cap Q'):w)\cap K[x_1,\ldots,x_r,x_{t+1},\ldots,x_n]$ is still an  intersection of two irreducible ideals and  $$\sdepth  (((Q\cap Q'):w)\cap K[x_1,\ldots,x_r,x_{t+1},\ldots,x_n])\geq \lceil\frac{r}{2}\rceil+\lceil\frac{n-t}{2}\rceil$$
by Lemma \ref{ea}. We have $\sdepth  (Q\cap Q'\cap K[x_{r+1},\ldots,x_t])\geq 1$ and so
$$\sdepth  (Q\cap Q'\cap K[x_{r+1},\ldots,x_t])S\geq 1+n-t+r$$
by  \cite[Lemma 3.6]{HVZ}. Thus $$\sdepth (Q\cap Q')\geq \sdepth \mathcal{D'}\geq \lceil\frac{r}{2}\rceil+\lceil\frac{n-t}{2}\rceil.$$
 Note that the proof goes even when $0\leq r< t\leq n$ (anyway $\sdepth  Q\cap Q'\geq 1$ if $n=t$, $r=0$).
\end{proof}

\begin{Lemma}\label{lob}
 $$\sdepth (Q\cap Q')\geq n-p+\lceil\frac{r}{2}\rceil+\lceil\frac{p-t}{2}\rceil.$$
\end{Lemma}
\begin{proof} As usual we see that there are now $(n-p)$ free variables and it is enough to apply
\cite[Lemma 3.6]{HVZ} and Lemma \ref{lb}.
\end{proof}

\begin{Theorem}\label{Lob}
Let $Q$ and $Q'$ be two non-zero irreducible monomial ideals of $S$. Then
$$\sdepth_S (Q\cap Q')\geq
\dim(S/(Q+ Q'))+\lceil\frac{\dim(S/Q')-\dim(S/(Q+ Q'))}{2}\rceil+$$
$$\lceil\frac{\dim(S/Q)-\dim(S/(Q+ Q'))}{2}\rceil\geq
\lceil\frac{\dim(S/Q')+\dim(S/ Q)}{2}\rceil.$$
\end{Theorem}
\begin{proof} After  renumbering of variables, we may suppose as above that
\\$\sqrt{Q}=(x_1,\ldots,x_t)$, $\sqrt{Q'}=(x_{r+1},\ldots,x_p)$
for some integers $r,t,p$ with $1\leq r\leq t<p\leq n$, or  $0= r< t< p\leq n$, or $1\leq r\leq t=p\leq n$.
If $n=p$, $r=0$ then $\sqrt{Q}\subset\sqrt{Q'}$ and the inequality is trivial.
It is enough to apply Lemma \ref{lob} because $n-p=\dim(S/(Q+ Q'))$, $r=\dim(S/Q')-\dim(S/(Q+ Q'))$, $p-t=\dim(S/Q)-\dim(S/(Q+ Q')).$
\end{proof}

\begin{Remark} \label{tr}{\em If $Q,Q'$ are non-zero irreducible monomial ideals of $S$ with $\sqrt{Q}=\sqrt{Q'}$ then
we have $\sdepth_S (Q\cap Q')\geq1+\dim S/Q$.}
\end{Remark}

\begin{Example}
{\em Let $S=K[x_1,x_2]$, $Q=(x_1)$,$Q'=(x_1^2,x_2)$. We have $\sdepth(Q\cap Q')\geq $
$$\lceil\frac{\dim(S/Q')+\dim(S/ Q)}{2}\rceil=\lceil\frac{1+0}{2}\rceil=1$$
by the above theorem.
As $Q\cap Q'$ is not a principle ideal  its Stanley depth is   $<2$. Thus $\sdepth(Q\cap Q')=1.$}
\end{Example}

\begin{Example}
{\em Let $S=K[x_1,x_2,x_3,x_4,x_5]$, $Q=(x_1,x_2,x_3^2)$,$Q'=(x_3,x_4,x_5)$.
As $\dim(S/(Q+ Q'))=0$, $\dim S/Q =2$ and $\dim S/Q' =2$
we get $\sdepth(Q\cap Q')\geq $
$$\lceil\frac{\dim(S/Q')+\dim(S/ Q)}{2}\rceil=\lceil\frac{2+2}{2}\rceil=2$$
by the above theorem. Note also that $\sdepth(Q\cap Q'\cap K[x_1,x_2,x_4,x_5])=$
$$\sdepth(x_1x_4,x_1x_5,x_2x_4,x_2x_5)K[x_1,x_2,x_4,x_5]=3,$$ and
$$\sdepth(((Q\cap Q'):x_3)\cap K[x_1,x_2,x_4,x_5])=$$ $$\sdepth((x_1,x_2)K[x_1,x_2,x_4,x_5])=4-\lfloor\frac{2}{2}\rfloor=3,$$ by \cite{Sh}.
But $\sdepth(Q\cap Q')\geq 3$ because of the following Stanley decomposition
$$Q\cap Q'=x_1 x_4 K[x_1,x_4,x_5]\oplus x_1 x_5 K[x_1,x_2,x_5]\oplus x_2 x_4 K[x_1,x_2,x_4]\oplus x_2 x_5 K[x_2,x_4,x_5]$$
$$\oplus x_3^2 K[x_3,x_4,x_5]\oplus x_2 x_3 K[x_2,x_3,x_4]\oplus
x_1 x_3 K[x_1,x_2,x_3]\oplus x_1 x_3 x_4 K[x_1,x_2,x_4,x_5]\oplus$$ $$ x_1 x_3 x_5 K[x_1,x_3,x_5]\oplus x_2 x_3 x_5 K[x_2,x_3,x_4,x_5]\oplus x_1 x_2 x_4 x_5 K[x_1,x_2,x_4,x_5]\oplus$$ $$ x_1 x_3^2 x_4 K[x_1,x_3,x_4,x_5]\oplus x_1 x_2 x_3 x_5 K[x_1,x_2,x_3,x_5]\oplus x_1 x_2 x_3^2 x_4 K[x_1,x_2,x_3,x_4,x_5].$$}
\end{Example}

\vskip 0.5 cm
\section{Applications}

Let $I\subset S $ be a non-zero monomial ideal. A. Rauf presented  in \cite{R} the following
\begin{Question}\label{as} Does it hold the inequality
$$\sdepth I\geq 1+\sdepth S/I?$$
\end{Question}
The importance of this question is given by the following:
\begin{Proposition}\label{appl} Suppose that Stanley's Conjecture holds for cyclic $S$-modules and the above question has a positive answer
for all monomial ideals of $S$. Then  Stanley's Conjecture holds for all monomial ideals of $S$.
\end{Proposition}
For the proof note that $\sdepth I\geq 1+\sdepth S/I \geq 1+\depth S/I=\depth I$.
\begin{Remark}{\em In \cite{P} it is proved that Stanley's Conjecture holds for all multigraded cycle modules over $S=K[x_1,\ldots,x_5]$. If the above question has a positive answer then   Stanley's Conjecture holds for all monomial ideals of $S$. Actually this is true for all square free monomial ideals of $S$ as \cite{Po}
shows.}
\end{Remark}

We show that the above question holds for the intersection of two non-zero irreducible monomial ideals.
\begin{Proposition}\label{asia} Question \ref{as} has a positive answer for intersections of two  non-zero irreducible monomial ideals.
\end{Proposition}
\begin{proof} First suppose that $Q,Q'$ have different associated prime ideals.
After renumbering of variables we may suppose as above that $\sqrt{Q}=(x_1,\ldots,x_t)$, $\sqrt{Q'}=(x_{r+1},\ldots,x_p)$
for some integers $r,t,p$ with $1\leq r\leq t<p\leq n$, or  $0= r< t< p\leq n$, or $1\leq r\leq t=p\leq n$. Then
$$\sdepth (Q\cap Q')\geq n-p+\lceil\frac{r}{2}\rceil+\lceil\frac{p-t}{2}\rceil$$
by Lemma \ref{lob}. Note that $\sdepth(S/(Q\cap Q'))=$
 $$n-p+\max\{\min\{r,\lceil\frac{p-t}{2}\rceil\},\min\{p-t,\lceil\frac{r}{2}\rceil\}\}$$
by Corollary \ref{eg}. Thus $$1+ \sdepth(S/(Q\cap Q'))\leq n-p+\lceil\frac{r}{2}\rceil+\lceil\frac{p-t}{2}\rceil\leq \sdepth (Q\cap Q').$$
Finally, if $Q,Q'$ have the same associated prime ideal then $\sdepth (Q\cap Q')\geq 1+\dim S/Q$ by Remark \ref{tr} and so $\sdepth (Q\cap Q')\geq 1+\sdepth S/(Q\cap Q')$.
\end{proof}

Next we will show that Stanley's Conjecture holds for intersections of two primary monomial ideals. We start with a simple lemma.

\begin{Lemma} \label{sc}
Let $Q$,$Q'$ be two primary ideals in $S=K[x_1,\ldots,x_n]$. Suppose $\sqrt{Q}=(x_1,\ldots,x_t)$ and $\sqrt{Q}=(x_{r+1},\ldots,x_p)$ for integers $0\leq r\leq t\leq p\leq n$. Then $\sdepth(S/(Q\cap Q'))\geq \depth(S/(Q\cap Q'))$, that is Stanley's Conjecture holds for \\
$S/(Q\cap Q')$.
\end{Lemma}
\begin{proof}
If either $r=0$, or $t=p$ then $\depth S/(Q\cap Q')\leq n-p\leq \sdepth(S/(Q\cap Q'))$ by   \cite[Lemma 3.6]{HVZ}. Now suppose that $r>0$, $t<p$ and let
 $S'=K[x_1,\ldots,x_p]$ and $q=Q\cap S'$, $q'=Q'\cap S'$. Consider the following exact sequence of $S'$-modules
$$0\rightarrow S'/(q\cap q')\rightarrow S'/q \oplus S'/q'\rightarrow S'/(q+q')\rightarrow 0.$$
\\ By Lemma \ref{l2}
$$\depth(S'/q \oplus S'/q')=\min\{\depth(S'/q),\depth(S'/q')\}=$$
$$\min\{\dim(S'/q),\dim(S'/q')\}=
\min\{r,p-t\}\geq1>0=\depth(S'/(q+q')).$$
\\ Thus by Depth Lemma (see e.g. \cite{BH}) $$\depth(S'/q\cap q')=\depth(S'/(q+q'))+1=1.$$
But $\sdepth(S'/(q\cap q'))\geq1$ by \cite[Corollary 1.6]{C1} and so
$$\sdepth(S/(Q\cap Q'))=\sdepth(S'/(q\cap q'))+n-p\geq 1+n-p=$$
$$n-p+\depth(S'/(q\cap q')) =\depth(S/(Q\cap Q'))$$
by \cite[Lemma 3.6]{HVZ}.
\end{proof}

\begin{Theorem}\label{sci} Let $Q$,$Q'$ be two non-zero irreducible ideals of $S$.
Then \\ $\sdepth (Q\cap Q')\geq
\depth (Q\cap Q')$, that is
Stanley's Conjecture holds for $Q\cap Q'$.
\end{Theorem}
\begin{proof} By Proposition \ref{asia} the Question \ref{as} has a positive answer, so by the proof of Proposition \ref{appl} it is enough to know that
Stanley's Conjecture holds for $S/(Q\cap Q')$. This is given by the above lemma.
\end{proof}

Next we consider the cycle module given by an intersection of 3 irreducible ideals.

\begin{Lemma}\label{3}
Let $Q_1$,$Q_2$, $Q_3$ be three non-zero irreducible monomial ideals of $S=K[x_1,\ldots,x_n]$. Then
$$\sdepth((Q_2\cap Q_3)/(Q_1\cap Q_2\cap Q_3))\geq $$
$$\dim(S/(Q_1+ Q_2+ Q_3))+\lceil \frac{\dim(S/(Q_1+Q_2))-\dim(S/(Q_1+Q_2+Q_3))}{2}\rceil  +$$
$$\lceil \frac{\dim(S/(Q_1+Q_3))-\dim(S/(Q_1+Q_2+Q_3))}{2}\rceil\geq $$
$$\lceil \frac{\dim(S/(Q_1+Q_2))+\dim(S/(Q_1+Q_3))}{2}\rceil.$$
\end{Lemma}

\begin{proof}
Renumbering the variables we may assume that $\sqrt{Q_1}=(x_1,\ldots,x_t)$ and $\sqrt{Q_2+Q_3}=(x_{r+1},\ldots,x_p)$, where $0\leq r\leq t<p\leq n.$
If $t=p$ then $Q_1+Q_2=Q_1+Q_3$ and the inequality is trivial by \cite[Lemma 3.6]{HVZ}.
Let $S'=K[x_1,\ldots,x_p]$ and $q_1=Q_1\cap S'$, $q_2=Q_2\cap S'$, $q_3=Q_3\cap S'.$
 We have a canonical injective map $(q_2\cap q_3)/(q_1\cap q_2\cap q_3)\longrightarrow S'/q_1$. Now by Lemma \ref{l1}
 we have $$S'/q_1=\oplus uK[x_{t+1},\ldots,x_p]$$ and so
$$(q_2\cap q_3)/(q_1\cap q_2\cap q_3)=\oplus((q_2\cap q_3)\cap uK[x_{t+1},\ldots,x_p]),$$
where $u$ runs in the monomials of $K[x_1,\ldots,x_t]\setminus(q_1\cap K[x_1,\ldots,x_t]).$
 If \\
 $u\in K[x_1,\ldots,x_r]$ then $(q_2\cap q_3)\cap uK[x_{t+1},\ldots,x_p]=u(q_2\cap q_3\cap K[x_{t+1},\ldots,x_p])$ and if $u\not\in K[x_1,\ldots,x_r]$ then $(q_2\cap q_3)\cap uK[x_{t+1},\ldots,x_p]=u(((q_2\cap q_3):u)\cap K[x_{t+1},\ldots,x_p]).$ Since $(q_2\cap q_3):u$ is still an intersection of irreducible monomial ideals
 we get by Lemma \ref{lb} that $$\sdepth (((q_2\cap q_3):u)\cap K[x_{t+1},\ldots,x_p])\geq $$ $$\lceil\frac{\dim K[x_{t+1},\ldots,x_p]/q_2\cap K[x_{t+1},\ldots,x_p]}{2}\rceil+$$ $$\lceil\frac{\dim K[x_{t+1},\ldots,x_p]/q_3\cap K[x_{t+1},\ldots,x_p]}{2}\rceil
.$$

 Also we have
 $$q_2/(q_1\cap q_2)=\oplus u(q_2\cap K[x_{t+1},\ldots,x_p]),$$
 and it follows
 $$S'/(q_1+q_2)\cong (S'/q_1)/(q_2/(q_1\cap q_2))=\oplus u(K[x_{t+1},\ldots,x_p]/q_2\cap K[x_{t+1},\ldots,x_p]).$$
 Thus $\dim S'/(q_1+q_2)=\dim K[x_{t+1},\ldots,x_p]/q_2\cap K[x_{t+1},\ldots,x_p]$ and similarly
 \\$\dim S'/(q_1+q_3)=\dim K[x_{t+1},\ldots,x_p]/q_3\cap K[x_{t+1},\ldots,x_p].$ Hence
 $$\sdepth((q_2\cap q_3)/(q_1\cap q_2\cap q_3))\geq \lceil \frac{\dim(S'/(q_1+q_2))}{2}\rceil + \lceil \frac{\dim(S'/(q_1+q_3))}{2}\rceil =$$
 $$\lceil \frac{\dim(S/(Q_1+Q_2))-\dim(S/(Q_1+Q_2+Q_3))}{2}\rceil + $$
 $$\lceil \frac{\dim(S/(Q_1+Q_3))-\dim(S/(Q_1+Q_2+Q_3))}{2}\rceil .$$
Now it is enough to apply \cite[Lemma 3.6]{HVZ}.
\end{proof}
\begin{Proposition}\label{s31} Let $Q_1$,$Q_2$, $Q_3$ be three non-zero irreducible ideals of $S$ and $R=S/Q_1\cap Q_2\cap Q_3$. Suppose that $\dim S/(Q_1+Q_2+Q_3)=0$.
Then $\sdepth R\geq$
$$ \max\{\min\{ \sdepth S/(Q_2\cap Q_3), \lceil \frac{\dim(S/(Q_1+Q_2))}{2}\rceil+\lceil \frac{\dim(S/(Q_1+Q_3))}{2}\rceil\},$$
$$\min \{ \sdepth S/(Q_1\cap Q_3),  \lceil \frac{\dim(S/(Q_1+Q_2))}{2}\rceil+\lceil \frac{\dim(S/(Q_2+Q_3))}{2}\rceil\},$$
$$\min \{ \sdepth S/(Q_1\cap Q_2),  \lceil \frac{\dim(S/(Q_3+Q_2))}{2}\rceil+\lceil \frac{\dim(S/(Q_1+Q_3))}{2}\rceil\}\}.$$
\end{Proposition}
For the proof apply Lemma \ref{l3} and Lemma  \ref{3}.

\vskip 0.5 cm
\begin{Theorem}\label{s3} Let $Q_1$,$Q_2$, $Q_3$ be three non-zero irreducible ideals of $S$  and $R=S/(Q_1\cap Q_2\cap Q_3)$. Then $\sdepth R\geq \depth R$, that is
Stanley's Conjecture holds for  $R$.
\end{Theorem}

\begin{proof}  Applying \cite[Lemma 3.6]{HVZ}
we may reduce the problem to the case when $\dim S/(Q_1+Q_2+ Q_3)=0.$ If one of the $Q_i$ has dimension $0$ then $\depth R=0$ and there exists nothing to show. Assume that all $Q_i$ have dimension $>0$. If one of the $Q_i$ has dimension $1$ then $\depth R=1$ and by \cite{C1} (or \cite{C3}) we get $\sdepth R\geq 1=\depth R$. From now on
we assume that all $Q_i$ have dimension $>1$.

If $Q_1+Q_2$  has dimension $0$ then from the exact sequence
$$0\to R\to S/Q_1\oplus S/Q_2\cap Q_3\to S/(Q_1+Q_2)\cap (Q_1+Q_3)\to 0,$$
we get $\depth R=1$ by Depth Lemma  and we may apply \cite{C1} (or \cite{C3}) to get as above $\sdepth R\geq 1=\depth R$.
Thus we may suppose that $Q_1+Q_2$, $Q_2+Q_3$, $Q_1+Q_3$ have dimension $\geq 1$. Then from the exact sequence
$$0\to S/(Q_1+Q_2)\cap (Q_1+Q_3)\to S/(Q_1+Q_2)\oplus S/(Q_1+Q_3)\to S/(Q_1+Q_2+Q_3)\to 0$$
we get by Depth Lemma $\depth S/(Q_1+Q_2)\cap (Q_1+Q_3)=1.$
Renumbering $Q_i$ we may suppose that $\dim (Q_2+Q_3)\geq\max\{\dim (Q_1+Q_3),\dim (Q_2+Q_1)\}.$
Using Proposition \ref{s31} we have
$$\sdepth R\geq \min\{\sdepth S/Q_2\cap Q_3, \lceil \frac{\dim(S/(Q_1+Q_2))}{2}\rceil + \lceil \frac{\dim(S/(Q_1+Q_3))}{2}\rceil \}
 .$$
We may suppose that $\sdepth R<\dim S/Q_i$ because otherwise  $\sdepth R\geq\dim S/Q_i$ $\geq \depth R$. Thus using Theorem \ref{Low} we get
$$\sdepth R\geq \min\{\lceil \frac{\dim S/Q_3+\dim S/(Q_2+ Q_3)}{2}\rceil,$$
$$\lceil \frac{\dim(S/(Q_1+Q_2))}{2}\rceil +
\lceil \frac{\dim(S/(Q_1+Q_3))}{2}\rceil \}
 .$$
If $Q_1\not \subset \sqrt{Q_3}$ then $\dim S/Q_3>\dim S/(Q_1+Q_3)$ and we get $$\dim S/Q_3+\dim S/(Q_2+ Q_3)>\dim(S/(Q_1+Q_2))+\dim(S/(Q_1+Q_3))$$
because $\dim S/(Q_2+ Q_3)$ is maxim by our choice. It follows that
$$\sdepth R\geq  \lceil \frac{\dim(S/(Q_1+Q_2))}{2}\rceil +
\lceil \frac{\dim(S/(Q_1+Q_3))}{2}\rceil \geq 2
 .$$
But from the first above exact sequence we get $\depth R=2$ with Depth Lemma, that is $\sdepth R\geq \depth R$.

If $Q_1\not \subset \sqrt{Q_2}$ we note that $\dim S/Q_2+\dim S/(Q_2+ Q_3)>\dim(S/(Q_1+Q_2))+\dim(S/(Q_1+Q_3))$ and we proceed similarly as above with $Q_2$ instead $Q_3$.
Note also that if $Q_1 \subset \sqrt{Q_2}$ and $Q_1 \subset \sqrt{Q_3}$ we get $\dim S/(Q_2+Q_3)\geq  \dim S/(Q_2+Q_1)=\dim S/Q_2$, respectively
 $\dim S/(Q_2+Q_3)\geq  \dim S/(Q_3+Q_1)=\dim S/Q_3$. Thus $Q_1\subset \sqrt{Q_3}=\sqrt{Q_2}$ and it follows $\sdepth R\geq \dim S/Q_2$, which is a contradiction.
\end{proof}

\end{document}